\documentclass[12pt]{amsart}

\usepackage{amssymb}
\usepackage{amsthm}
\usepackage{amsmath}
\usepackage{amsxtra}
\usepackage{mathrsfs}
\usepackage{enumerate}
\usepackage{fullpage}
\usepackage{colonequals}
\usepackage[all]{xy}
\usepackage{rotating}

\numberwithin{equation}{section}

\newenvironment{problab}[1]
{\noindent\textbf{#1}.}
{\vskip 6pt}

\newtheorem{theorem}[equation]{Theorem}

\newtheorem{lemma}[equation]{Lemma}

\theoremstyle{remark}
\newtheorem{remark}[equation]{Remark}

\newcommand{\M}{\mathfrak M}

\newcommand{\F}{\mathbb F}

\newcommand{\Q}{\mathbb Q}

\newcommand{\Or}{\mathcal{O}}
\newcommand{\Z}{\mathbb Z}

\newcommand{\p}{\mathfrak{p}}

\newcommand{\sidesim}{\begin{sideways}%
      $\sim$\end{sideways}}

\DeclareMathOperator{\subexp}{subexp}

\DeclareMathOperator{\Id}{Id}

\DeclareMathOperator{\nrd}{nrd}

\DeclareMathOperator{\rad}{rad}
\DeclareMathOperator{\Sym}{sym}
\DeclareMathOperator{\charpoly}{charpoly}
\DeclareMathOperator{\minpoly}{minpoly}

\def\id{\operatorname{Id}}
\def\id{\operatorname{id}}

\def\det{\operatorname{det}}

\def\exp{\operatorname{exp}}

\def\M{\operatorname{M}}

\def\rank{\operatorname{rank}}

\def\det{\operatorname{det}}

\def\GL{\operatorname{GL}}
\def\M{\operatorname{M}}

\begin{document}

\title{Metacommutation of primes in central simple algebras}
\author{Sara Chari}

\maketitle

\begin{abstract}

In a quaternion order of class number one, an element can be factored in multiple ways depending on the order of the factorization of its reduced norm. The fact that multiplication is not commutative causes an element to induce a permutation on the set of primes of a given reduced norm. We discuss this permutation and previously known results about the cycle structure, sign, and number of fixed points for quaternion orders. We generalize these results to other orders in central simple algebras over global fields.

\end{abstract}


\section{Introduction}

Let $B=(-1,-1 \mid \Q)$ be the Hamilton quaternion algebra over $\Q$ and let $\Or= \Z+ \Z i + \Z j + \Z \left[ \frac{-1+i+j+ij}{2}\right] \subseteq B$ be the Hurwitz order. Given an element $\alpha \in \Or$ with reduced norm $\nrd(\alpha) =a\not=0$ and a factorization $a=p_1 p_2\cdots p_r$ (where the $p_i$ are prime, not necessarily distinct), there exists
a factorization $\alpha=\pi_1 \pi_2 \cdots \pi_r$, where $\pi_i \in \Or$ and $\nrd(\pi_i)=p_i$ for $i=1, \dots, r$. If $\alpha$ is primitive (i.e., $\alpha$ is not divisible in $\Or$ by a positive integer $m \geq 2$), then this factorization of $\alpha$ is unique up to \emph{unit migration}: the only other such factorizations are of the form $$\alpha=(\pi_1\varepsilon_1)(\varepsilon_1^{-1}\pi_2\varepsilon_2)\cdots (\varepsilon_{r-1}^{-1}\pi_r),$$ where each $\varepsilon_i \in \Or^\times$. For further reading, see Conway--Smith \cite[Chapter 5]{ConwaySmith} and Voight \cite[Chapter 11]{Voight}.

Because $\Or$ is not commutative, the factorization of $\alpha$ depends on the order of the primes $p_1, \dots, p_r$ in the factorization of $a$. It is therefore of interest to study how switching the order of the prime factors of $a$ affects the factorization of $\alpha$. For simplicity, consider the case where $a=pq$, where $p$ and $q$ are distinct primes. We may then factor $\alpha=\pi \omega$, where $\nrd(\pi)=p$ and $\nrd(\omega)=q$ as above. However, if we factor $a=qp$, then we obtain a different factorization $\alpha=\omega'\pi'$, where $\nrd(\omega')=q$ and $\nrd(\pi')=p$. These factorizations are unique (up to unit migration) as described by Conway--Smith \cite[Chapter 5]{ConwaySmith}, so in particular, $\pi'$ is unique up to left multiplication by units in $\Or$. In this way, $\omega$ induces a permutation $\sigma_\omega$ on the set of elements $\{\pi_i\}_i$ of reduced norm $p$ up to left multiplication by units as follows: if $\pi_i\omega=\omega'\pi_j$ for some $\omega' \in \Or$, then we define $\sigma_\omega(\pi_i)\colonequals\pi_j$. 

Cohn--Kumar \cite{CohnKumar} studied the permutation $\sigma_\omega$. They computed the number of fixed points and the sign of $\sigma_\omega$.
%
%
Their result was reproven by Forsyth--Gurev--Shrima \cite{FGS}, who showed that the permutation $\sigma_\omega$ can be understood through an action of $\GL_2(\F_p)$ on $\mathbb{P}^1(\F_p)$ via a correspondence of Hurwitz primes with points on a conic. They also determined the cycle structure of $\sigma_\omega$ by doing so for the corresponding permutation $\sigma_Q$.

In this paper, we extend the results of Forsyth--Gurev--Shrima to a more general setting. Let $R$ be a Dedekind domain with field of fractions $K$ and let $\p \subseteq R$ be a prime ideal. Let $\Or$ be an $R$-order in a central simple algebra $B$ over $K$ and define $\Id(\Or;\p)$ to be the set of left ideals $I \subseteq \Or$ of reduced norm $\p$. Let $K_{(\p)}$ be the localization at $\p$ and $K_\p$ the completion at $\p$. Let $\Or_{(\p)} \subseteq B$ denote the localization of $\Or$ at $\p$ and $\Or_{\p} \subseteq B_\p$ the completion of $\Or$ at $\p$. In Section 2, we define in an analogous way a permutation $\sigma_{\omega}$ of the set $\Id(\Or;\p)$, induced by an element $\omega$ in $\Or_{(\p)}^\times \cap \Or$ in the case where $\Or_\p$ is maximal in $B_\p$. We then show in Section 3 that by reduction of $\Or_{\p}$ modulo its Jacobson radical $J_{\p}$ and by identifying elements of $\Id(\Or;\p)$ with the kernels of the corresponding elements of $\Id(\Or_\p; \p)$ modulo $J_{\p}$, we get a group action of $(\Or_{\p}/J_{\p})^\times \simeq \GL_m(\F_q)$ on $\mathbb{P}^{m-1}(\F_q)$ for some $m \in \Z_{>0}$ and $\F_q$ a finite field extension of $R/\p$, by $Q \cdot v = Q^{-1}v$. For each $Q \in \GL_m(\F_q)$, this action gives a permutation $\tau_Q$ of $\mathbb{P}^{m-1}(\F_q)$ by $\tau_Q(v)=Q \cdot v$. Let $\rho \colon \Or_{\p}^\times \rightarrow \GL_m(\F_q)$ be the reduction map modulo $J$, and let $\sigma \colon \Or_{\p}^\times  \rightarrow \Sym(\Id(\Or_\p;\p))$ and $\tau \colon \GL_m(\F_q) \rightarrow \Sym(\mathbb{P}^{m-1}(\F_q))$ be the maps that send $\omega$ to $\sigma_\omega$ and $Q$ to $\tau_Q$, respectively. In section 4, we prove our main result.

\begin{theorem} Given $\omega \in \Or_{\p}^\times$ and $Q = \rho(\omega)$, then the following diagram commutes.

$$\xymatrix{
{\Or_{\p}^\times}\ar^{\sigma\hspace{.2in}}[r]\ar_{\rho}@{^{}->}[d]&{\Sym(\Id(\Or_\p;\p))}\ar^{\sidesim}[d]\\
{\GL_m(\F_q)}\ar^{\tau\hspace{.2in}}[r]&{\Sym(\mathbb{P}^{m-1}(\F_q))}
}$$

\end{theorem}

This theorem provides a way to understand the permutation $\omega \in \Or_{(\p)}^\times \cap \Or$ through the action of $\GL_m(\F_q)$ on $\mathbb{P}^{m-1}(\F_q)$ in the general setting of central simple algebras. Under certain conditions, the cycle structure of $\sigma_\omega$ is easily determined in higher dimensions, and in these cases, we compute the cycle structure in Section 5. This answers Conway and Smith's original question in \cite[Chapter 5.5]{ConwaySmith} about how the factorization $\pi \omega$ relates to the factorization $\omega'\pi'$ via the cycle structure of $\sigma_\omega$ and rephrases the question in the general context of central simple algebras. While the exact structure of the permutation $\sigma_\omega$ is not as easily seen from the trace and norm of $\omega$, we are still able to view the permutation as one given by an action of matrices on projective space. Results of Fripertinger \cite{HF} can then be used to compute the cycle structure of the permutation given by the matrix action for larger matrices.

Since our permutations arise from a lack of commutativity, the size of the cycles of $\sigma_\omega$ can be interpreted as a way to determine how close $\omega$ is to the center of the completion $\Or_\p$; i.e. if all left ideals are principal, then starting with $\pi \in \Or$, the size of a cycle containing an element $\pi$ up to left multiplication by units is the number of times we must apply metacommutation by $\omega$ in order to obtain $\pi$ again.

We discuss the factorization as it applies to matrix algebras in Example \ref{ex1}. This study of factorization in noncommutative rings has been studied in several different contexts, including \cite{WR} where Rump describes the set of Hurwitz primes (elements of prime reduced norm) as an $L^*$-algebra and shows that metacommutation can occur in any $L^*$-algebra. He also describes the relationship of metacommutation with certain Garside groups. 



We thank John Voight and Daniel Smertnig for their helpful comments and feedback.

\section{Metacommutation in central simple algebras}


In this section, we set up notation and define the permutation $\sigma_\omega$. Let $R$ be a Dedekind domain whose field of fractions is a global field $K$. Let $B$ be a finite-dimensional central simple algebra of dimension $n^2$ over $K$ and let $\Or \subseteq B$ be an $R$-order. For a prime ideal $\p \subseteq R$, define $R_{(\p)}$ and $R_\p$ to be the localization and completion at $\p$, respectively. Define $\Or_{(\p)} \colonequals \Or \otimes_R R_{(\p)}$ and $\Or_\p \colonequals \Or \otimes_R R_\p$ to be the localization and completion of $\Or$ at $\p$. We will choose $\p \subseteq R$ such that $\Or_{\p}$ is maximal in $B_\p$.


To motivate a more general construction, we will first consider the case where all left $\Or$-ideals are principal and we call $\Or$ a left principal ideal ring (PIR). To define the permutation $\sigma_\omega$ in this case, we use the following theorem on factorization in $\Or$. 

\begin{theorem}\label{thm0} Let $\alpha \in \Or$ and write $\nrd(\alpha)=a_1a_2$ with $Ra_1+Ra_2=R$. Then, we can factor $\alpha=\omega_1\omega_2$ with $\omega_1, \omega_2 \in \Or$, where $\nrd(\Or\omega_1)=Ra_1$ and $\nrd(\Or\omega_2)=Ra_2$. Moreover, if $a_2 \in R$ is prime, then $\omega_2$ is unique up to left multiplication by elements in $\Or^\times$. 
\end{theorem}

\begin{proof}

Consider the left ideal $I=\Or \alpha+ \Or a_2$. Then, $I=\Or \omega_2$ for some $\omega_2 \in \Or$. Note also that $\alpha \in I$, so $\alpha=\omega_1\omega_2$ for some $\omega_1 \in \Or$. One can verify that $\nrd(I)=Ra_2$. 
%
%
%

To conclude, we show that if $a_2=\nrd(\omega_2)$ is prime, then $\omega_2$ is unique up to left multiplication by units. This is because if we have $\alpha=\omega_1\omega_2=\omega_1'\omega_2'$, then $\Or \omega_2 \subseteq \Or \omega_2 + \Or \omega_2'=\Or \pi$ for some $\pi \in \Or$ since $\Or$ is a PIR.  It follows that $\omega_2=\beta \pi$ for some $\beta \in \Or$ and similarly, $\omega_2'=\beta'\pi$ for some $\beta' \in \Or$. We must have either $\beta \in \Or^\times$ or $\pi \in \Or^\times$ since $a_2=\nrd(\omega_2)=\nrd(\beta)\nrd(\pi)$ is prime and hence irreducible. But, $a_2 \mid \nrd(\omega_2)$ and $a_2 \mid \nrd(\omega')$, so $a_2 \mid\nrd(\pi)$. Therefore, $\nrd(\pi) \notin R^\times$, so $\nrd(\beta) \in R^\times$, so $\beta\in \Or^\times$ and similarly, $\beta' \in \Or^\times$. Finally, $\omega_2=\beta^{-1}\beta' \omega_2'$, so $\omega_2$ is unique up to left multiplication by units in $\Or$. \qedhere 
\end{proof}

\begin{lemma} \label{lem2}

Let $\omega \in \Or_{(\p)}^\times \cap \Or$ be an element of reduced norm $a$. Then for each element $\pi \in \Or$ of prime reduced norm $\nu$, we obtain a new element of reduced norm $u\nu$ for $u \in R^\times$, which is unique up to left multiplication by units in $\Or^\times$. 

\end{lemma}

\begin{proof}
Let $\alpha=\pi \omega$ and factor $\nrd(\alpha)=\nrd(\pi \omega) =a\nu$. Then, by Theorem \ref{thm0}, we can factor $\alpha=\omega'\pi'$ where $\omega'$ and $\pi'$ have the desired reduced norms and $\pi'$ is unique up to left multiplication by units in $\Or^\times$. 
\end{proof}

Let $\Id(\Or;R\nu)$ be the set of elements in $\Or$ of reduced norm $\nu$, up to left multiplication by units. We define a map $\sigma_\omega \colon \Id(\Or; R\nu) \rightarrow \Id(\Or; R \nu)$ by $$\pi \mapsto \pi'$$ if $$\pi \omega=\omega'\pi'$$ for some $\omega' \in \Or$.



\section{The Permutation}
We now consider the more general case where we no longer require the left ideals of $\Or$ to be principal. Let $P \subseteq \Or$ be a left ideal of reduced norm $\nrd(P)=\p$ and let $\omega \in \Or_{(\p)}^\times \cap \Or$. 
We define a new left ideal of $\Or$ depending on $P$ and $\omega$, by
\begin{equation} \label{eq1}
P'\colonequals P \omega+\Or \p.
\end{equation}
\begin{lemma} The set $P'$ in \eqref{eq1} is a left $\Or$-ideal with $\nrd(P')=\p$.
\end{lemma}
\begin{proof} 
First, $P'$ is finitely generated because $P$ is an ideal and hence a lattice, which is finitely generated, and $\p$ is finitely generated over $R$ since $R$ is Noetherian. Therefore, $P\omega$ and $\Or \p$ are also finitely generated. We also have $\Or \p \subseteq P'$ and $\Or \p K=\Or K=B \subseteq \Or P'$, so $P'K=B$. Therefore, $P'$ is an $\Or$-lattice in $B$. Finally, we have $\Or \subseteq \Or_L(P)$, so $\Or \subseteq \Or_L(P \omega)$. Also, $\Or =\Or_L(\Or \p)$ and so $\Or \subseteq \Or_L(P \omega) \cap \Or_L(\Or \p) \subseteq \Or_L(P\omega+\Or \p)=\Or_L(P')$. Thus, $P'$ is a left $\Or$-ideal. 

To show that $P'$ has reduced norm $\p$, first note that $P \omega \subseteq P'$, so $\nrd(P\omega)\subseteq \nrd(P')$. We also have $\nrd(P \omega)=\p\nrd(\omega),$ so $\p\nrd(\omega) \subseteq \nrd(P').$ Similarly, $\p \subseteq P'$, so $\nrd(\p)=R \p^n \subseteq \nrd(P')$. Then, $\p^n + \p\nrd(\omega)=\p \subseteq \nrd(P')$. Finally, $\nrd(P')=\nrd(P \omega+\Or \p) \subseteq \nrd(P \omega)+\nrd (\Or \p)=\p \nrd(\omega)+ \p^n=\p$, so equality holds and $\nrd(P')=\p.$ \qedhere 
\end{proof}

Now, let $\Id(\Or;\p)$ be the set of ideals in $\Or$ of reduced norm $\p$ and define the map $$\sigma_\omega \colon \Id(\Or;\p) \rightarrow \Id(\Or;\p)$$ $$P \mapsto P\omega+\Or \p.$$ When both definitions of $\sigma_\omega$ are relevant, then they are the same, as shown in the following lemma.

\begin{lemma} If all left ideals of $\Or$ are principal, then the map $\sigma_\omega$ is the map obtained via metacommutation by $\omega$ described in Lemma \ref{lem2}. 
\end{lemma}

\begin{proof} Define $P'=P\omega+\Or \p$ as before. If all left ideals of $\Or$ are principal, then $P=\Or \pi$ and $P'=\Or \pi'$ for some $\pi, \pi' \in \Or$ with $R\nrd(\pi)=R\nrd(\pi')=\p$. Then, $\pi \omega \in P'$, so $\pi \omega \in \Or \pi'$. Finally, we have $\pi \omega =\omega' \pi'$ for some $\omega' \in \Or$, and we recover the permutation given by metacommutation. \qedhere 
\end{proof}

We note that we only needed the ideals in $\Id(\Or;\p)$ to be principal in order to obtain the factorization described in Lemma \ref{lem2} and the following definition of $\sigma_\omega$.

\begin{theorem}
The map $\sigma_\omega$ is a permutation of the set $\Id(\Or;\p)$. 
\end{theorem}

\begin{proof}
We will show that $\omega \in \Or_{(\p)}^\times \cap \Or$, $\sigma_\omega$ is a bijection by producing an inverse map. First, we show that for $\omega_1, \omega_2 \in \Or_{(\p)}^\times \cap \Or$, we have $\sigma_{\omega_1}\sigma_{\omega_2}=\sigma_{\omega_2\omega_1}$. Now, $\sigma_{\omega_1}\sigma_{\omega_2}(P)=\sigma_{\omega_1}(P\omega_2+\Or \p)=(P\omega_2+\Or \p)\omega_1+\Or \p=P\omega_2\omega_1+\Or \p \omega_1+\Or \p=P \omega_2\omega_1+\Or \p=\sigma_{\omega_2\omega_1}$. Also, there is an element $b\in R$ such that $b\omega^{-1} \in \Or$. Then, $\sigma_{\omega}\sigma_{b\omega^{-1}}=\sigma_b=Pb+\Or \p=P$. Therefore, $\sigma_\omega$ is a permutation whose inverse is given by $\sigma_{b\omega^{-1}}$. 
\end{proof}

\section{An action of matrices on projective space}

We now describe the permutation $\sigma_\omega$ as a group action of $\GL_m(\F_q)$ on $\mathbb{P}^{m-1}(\F_q)$. We will do so via the completion of $\Or$ at $\p$, and by reduction modulo the Jacobson radical. The map $I \mapsto I_\p \colonequals I \otimes_R R_\p$ is a bijection between $\Id(\Or;\p)$ and the ideals of reduced norm $\p R_{\p}$ in $\Or_{\p}$ by the local-global dictionary for lattices and \cite[Theorem 5.2(iii),5]{R}, where $\Or_{\p} \simeq \Or \otimes R_{\p}$ is the completion at $\p$. Then, $P_\p \mapsto P_\p'=P_\p\omega+\Or_\p \p$ if and only if $P=P_\p \cap \Or \mapsto (P_\p \omega + \Or_\p \p) \cap \Or=P \omega+\Or \p=P'$. It therefore suffices to study the local case, and we recover the global case through this correspondence. To simplify notation, throughout this section, let $R$ be the valuation ring of a local field $K$ and let $\p$ be its unique maximal ideal. Define $\F_\p \colonequals R/\p$. Let $B$ be a central simple $K$-algebra and let $\Or$ be an $R$-order in $B$. We then consider the action of $\Or^\times$ on $\Id(\Or;\p)$ given by $\omega \cdot P=P\omega+\Or \p$.

\subsection*{A correspondence of ideals with matrices}
Let $J \colonequals \rad \Or$ be the Jacobson radical of $\Or$. The left ideals of $\Or$ with reduced norm $\p$ correspond to their images under the reduction map  modulo $J$ \cite{R}. By the Wedderburn-Artin Theorem, since $B$ is a central simple algebra over $K$, there is an isomorphism $B \simeq \M_m(D)$ of $K$-algebras for some $m \in \Z_{>0}$, where $D$ is a division algebra over $K$ of dimension $t^2$ and $n=tm$. Now, $D$ contains a unique maximal order $\Lambda$ with Jacobson radical $J_\Lambda \colonequals \rad \Lambda$, and $\Or \simeq \M_m(\Lambda)$ \cite[Section 5.17]{R}. We then have the following lemma.

\begin{lemma}
Let $J$ be the Jacobson radical of $\Or$. Then, there is an isomorphism $\Or/J \simeq \M_m(\F_q)$, where $\F_q$ is a finite extension of $\F_\p$. 
\end{lemma}
 
\begin{proof} There is an isomorphism $\Or/J \simeq \M_m(\Lambda/J_\Lambda)$ as proven by Reiner \cite[Section 5.17]{R}. The algebra $\Lambda/J_\Lambda$ is a finite-dimensional division algebra over $\F_\p=R/\p$. By Wedderburn's little theorem, $\Lambda/J_\Lambda \simeq \F_q$, where $\F_q$ is a field extension of $\F_\p$. We then have $\Or/ J\simeq \M_m(\F_{q})$.
\end{proof} 

Define $\rho \colon \Or \rightarrow \Or/J \simeq \M_m(\F_q)$ be the reduction map modulo $J$. Define $M_m^{(m-1)}$ to be the set of elements of rank $m-1$ in $\M_m(\F_q)$ up to left multiplication by units. We then have the following correspondence.

\begin{lemma}
The maximal left ideals of $\Or$ correspond to the maximal left ideals of $\M_m(\F_q)$, by $P \mapsto \rho(P)$. 
\end{lemma}

\begin{proof}
See Reiner \cite[Section 5.17]{R}.
\end{proof}

\begin{lemma} \label{lem1}
The map $\Id(\Or;\p) \mapsto M_m^{(m-1)}$ given by $P \mapsto \rho(P)$ is a bijection for some $m \in \Z_{>0}$, where $\F_q$ is a finite extension of $\F_\p$. 
\end{lemma}

\begin{proof} By the preceding lemma, the maximal left ideals of $\Or$ (those of reduced norm $\p$) correspond to the maximal left ideals of $\Or/J \simeq \M_m(\F_q)$ by mapping $P$ to $\rho(P)$. We now show that a left ideal $I \subseteq \M_n(\F_q)$ is maximal if and only if it is generated by an element of rank $m-1$. 

If $I$ is maximal, let $A \in I$ be an element of maximal rank $r$, which we can take to be in reduced row echelon form. Let $B \notin \M_m(\F_q)A$ be arbitrary. By elementary row operations, there is a matrix $E$ such that $\rank(EB)=1$ and the only nonzero row of $EB$ is the $m$th row which is not in the rowspan of $A$. We then have $\rank(A+EB)=r+1$ and $A+EB \notin I$ by maximality of $r$. If follows that $B \notin I$ so $I=\M_m(\F_q)A$ since $B$ was arbitrary. To show that $r=m-1$, note that $I=\M_m(\F_q)A \subset \M_m(\F_q)(A+EB)=\M_m(\F_q)$ by maximality of $I$, so $A+EB \in \GL_m(\F_q)$ and $r+1=m$. 

Conversely, if $I=\M_m(\F_q)A$ is generated by an element $A$ of rank $m-1$, let $I'\subseteq \M_m(\F_q)$ be a left ideal such that $I \subset I'$. Choose $E \in \M_m(\F_q)$ and $B \notin I$ as before, so $\rank(A+EB)=r+1=m$ and $A+EB \in \GL_m(\F_q)$. Then $I'=\M_m(\F_q)$ so $I$ is maximal. \qedhere 

\end{proof}


\subsection*{An action of matrices}
We now describe the action of $\Or^\times$ on $\Id(\Or;\p)$ as an action of $\GL_m(\F_q)$ on $\mathbb{P}^{m-1}(\F_q)$ via the bijection $P \leftrightarrow \rho(P) \leftrightarrow \ker(P)$. 

Let $P \in \Id(\Or;\p)$. By Lemma \ref{lem1}, $\rho(P)=\M_m(\F_q)A$ with $A$ of rank $m-1$. Since $P \mapsto P'=P \omega+\Or \p$, we have $\rho(P')=\M_m(\F_q)AQ=\M_m(\F_q) A'$, where $A'$ has rank $m-1$ and hence $AQ=Q'A'$ for some $Q' \in \GL_m(\F_q)$. For each $Q \in \GL_m(\F_q)$, define a map $$\sigma_Q \colon M_m^{(m-1)} \rightarrow M_m^{(m-1)}$$ $$A \mapsto AQ.$$ It is well-defined for $C A \mapsto C AQ \sim AQ$ for $C \in \GL_m(\F_q)$, and it is a bijection because its inverse is given by $\sigma_Q^{-1}(A)=AQ^{-1}$. 

\begin{theorem}\label{thm2}
The map $\sigma_\omega$ is a permutation of the set $\Id(\Or;\p)$ induced by $\omega \in \Or^\times$ and is the same as the permutation $\sigma_Q$ by identifying each $P$ with the set of generators of $\rho(P)$; i.e. the following diagram commutes.

$$\xymatrix{
{\Or^\times}\ar^{\sigma\hspace{.35in}}[r]\ar_{\rho}@{^{}->}[d]&{\Sym(\Id(\Or;\p))}\ar^{\sidesim}[d]\\
{\GL_m(\F_q)}\ar^{\tau\hspace{.15in}}[r]&{\Sym(M_m^{m-1})}
}$$

\end{theorem}

\begin{proof}
For an ideal $P \subseteq \Or$ of reduced norm $\p$, we have $P'=P\omega+\Or \p$. Let $A$ and $A'$ be generators of $\rho(P)$ and $\rho(P)'$, respectively as defined above, and let $Q =\rho(\omega) \in \GL_m(\F_q)$ be the image of $\omega$ under the map of reduction modulo $J$. We will show that if $\sigma_\omega(P)=P'$, then $\sigma_{Q}(A)=A'$, for then $\sigma_\omega=\sigma_{Q}$ by identifying $P$ with $A$ for each $P$. Since $\p \subseteq J$, we have $\M_m(\F_q)A'=\M_m(\F_q)AQ$. Therefore, $AQ=Q'A'$ for some $Q' \in \GL_m(\F_q)$, so $\sigma_{Q}(A)=A'$. Finally, $\sigma_\omega$ is a permutation and is given by $\sigma_{Q}$, by the identification of $\Id(\Or;\p)$ with $M_m^{(m-1)}$. \qedhere 
\end{proof}

\begin{remark} If $B$ is a division algebra, then $\Or/J \simeq \F_q$ is a finite field, and hence the only permutation induced by an element $\omega \in \Or^\times$ by metacommutation is the identity permutation. On the other hand, if $B \simeq \M_n(K)$, then $\Lambda=R$ and $J=\p$, so $\F_q=R/\p=\F_\p$ and $\Or/J \simeq \M_n(\F_\p)$. 
\end{remark}

\subsection*{Metacommutation as an action of matrices on projective space}

It now suffices to study $\sigma_Q$ for $Q \in \GL_m(\F_q)$. The kernel of an element in $M_m^{(m-1)}$ is well-defined, since multiplying on the left by a unit does not change it. Since the matrices in $M_m^{(m-1)}$ have rank $m-1$, their kernels are one-dimensional and can be viewed as elements of $\mathbb{P}^{m-1}(\F_q)$. We will use $A$ and $A'$ to denote matrices of rank $m-1$. Define an action of $\GL_m(\F_q)$ on $\mathbb{P}^{m-1}(\F_q)$ by $Q \cdot v=Q^{-1}v$. Define $\tau_Q$ to be the corresponding permutation; i.e., $\tau_Q(v)\colonequals Q^{-1}v$.

\begin{theorem} \label{thm3}
The action of $\GL_m(\F_q)$ on the set $M_m^{(m-1)}$ is equivalent to the action of $\GL_m(\F_q)$ on $\mathbb{P}^{m-1}(\F_q)$; i.e. $\sigma_Q(P)=P'$ if and only if $\tau_Q(v)=v'$, where $v=\ker P$ and $v'=\ker P'$. 
\end{theorem}

\begin{proof}

The elements of $M_m^{(m-1)}$ are in bijection with their kernels, each of which contains one element up to scaling. These are exactly the elements of $\mathbb{P}^{m-1}(\F_q)$, for no two elements of $M_m^{(m-1)}$ have the same kernel, and every element of $\mathbb{P}^{m-1}$ is the kernel of a matrix of rank $m-1$. We will show that if $AQ=Q'A'$ with $v \in \ker(A)$, then $Q^{-1}v \in \ker(A')$.


Suppose that $AQ=Q'A'$ and let $v \in \ker(A)$. Now, we must show that $A'(Q^{-1}v)=0$. But, $Av=0$, so we have $0=Av=AQ(Q^{-1}v)=Q'A'(Q^{-1}v).$ Then, since $Q' \in \GL_m(\F_q)$, we also have $Q'^{-1}Q'A'(Q^{-1}v)=A'(Q^{-1}v)=0$. Therefore, $Q^{-1}v \in \ker(A')$. \qedhere 
\end{proof}

Let $\sigma \colon \Or^\times \rightarrow \Sym(\Id(\Or;\p))$ and $\tau \colon \GL_m(\F_q) \rightarrow \Sym(\mathbb{P}^{m-1}(\F_q))$ be the maps that send $\omega$ to $\sigma_\omega$ and $Q$ to $\tau_Q$, respectively. The preceding lemmas lead us to the following main result.

\begin{theorem} Given $\omega \in \Or^\times$ and $Q = \rho(\omega)$, then the following diagram commutes.

$$\xymatrix{
{\Or^\times}\ar^{\sigma\hspace{.35in}}[r]\ar_{\rho}@{^{}->}[d]&{\Sym(\Id(\Or;\p))}\ar^{\sidesim}[d]\\
{\GL_m(\F_q)}\ar^{\tau\hspace{.2in}}[r]&{\Sym(\mathbb{P}^{m-1}(\F_q))}
}$$

\end{theorem}

\begin{proof}
This follows immediately from Theorems \ref{thm2} and \ref{thm3}.
\end{proof}

The following example demonstrates the use of the techniques discussed in this section to write an $n \times n$ as a product of prime-determinant matrices in multiple ways.\\

\begin{problab}{Example 1}\label{ex1}
Let $B=\M_n(\Q)$, and let $\Or=\M_n(\Z)$. Then, for $\alpha \in B$, we have $\nrd(\alpha)=\det(\alpha)$. Suppose $\alpha \in \Or$ and that there is no positive integer $k$ that divides every entry of $\alpha$. For a prime $p_1$ dividing $\det(\alpha)$, we can factor $\det(\alpha)=a_1p_1$ for some $a \in \Z$. Then, by Lemma \ref{lem2}, we have $\alpha=A_1P_1$ for matrices $A_1,P_1 \in \M_n(\Z)$, with $\det(A_1)=a_1$ and $\det(P_1)=p_1$. Similarly, we can again write $a_1=a_2p_2$ and $A_1=A_2P_2$, with $\det(A_2)=a_2$ and $\det(P_2)=p_2$. Iterating this process, we obtain $\alpha=P_r \cdots P_1$, with $\det(P_i)=p_i$, with each $p_i$ prime. Since we can choose freely how to arrange the prime factors of $\det(\alpha)$, this allows for several different ways to encode the information given by the matrix $\alpha$, depending on what order we would like to factor $\det(\alpha)$. For example, let $$\alpha=\left(
\begin{array}{cc}
5 & 1\\
0 & 3
\end{array} \right)=\left(
\begin{array}{cc}
1 & 1\\
0 & 3
\end{array}\right) \left(
\begin{array}{cc}
5 & 0\\
0 & 1
\end{array} \right),$$ so $\nrd(\alpha)=\det(\alpha)=15$. Let $p=3, q=5$, $P=\left(
\begin{array}{cc}
1 & 1\\
0 & 3
\end{array}\right),$ and $Q=\left(
\begin{array}{cc}
5 & 0\\
0 & 1
\end{array} \right)$. Since $P$ corresponds to its kernel $v=\left(
\begin{array}{c}
1 \\
2
\end{array} \right)$ modulo 3, we know $\ker \rho(P')=Q^{-1}v=\left(
\begin{array}{c}
1\\
1
\end{array} \right)$ when we write $\alpha=Q'P'$ with $\det(Q')=5$ and $\det(P')=3$. Then, up to left units, $P' \equiv \left(
\begin{array}{cc}
1 & 2\\
0 & 0
\end{array} \right) \pmod 3$. We can choose $P'=\left(
\begin{array}{cc}
1 & 2\\
0 & 3
\end{array} \right)$, and solving $Q'P'=\alpha$, we have $Q'=\left(
\begin{array}{cc}
5 & -3\\
0 & 1
\end{array} \right)$.

\end{problab}

\section{Cycle structure of $\tau_Q$}

We now discuss the cycle structure of the permutation $\tau_Q$ for $Q \in \GL_m(\F_q)$, obtained by the action of $\GL_m(\F_q)$ on $\mathbb{P}^{m-1}$ discussed in the previous section. Note that the permutation $\tau_Q$ has the same as the cycle structure as $\tau_{Q^{-1}}=\tau_Q^{-1}$. It therefore suffices to study the cycle structure of $\GL_m(\F_q)$ acting on $\mathbb{P}^1(\F_q)$ by $Q \cdot v=Qv$ for $Q \in \GL_m(\F_q)$ and $v \in \mathbb{P}^{m-1}(\F_q)$. First, we count the number of fixed points.

\begin{theorem} Let $Q \in \GL_m(\F_q)$ for $n\geq2$ and let $\lambda_1, \dots, \lambda_s \in \overline{\F_q}$ be the eigenvalues of $Q$ with multiplicities $a_1, \dots, a_s$, respectively. Then, the number of fixed points of $\tau_Q$ is $\sum_{i: \lambda_i \in \F_q} \#\mathbb{P}^{a_i-1}(\F_q)=\sum_{i:\lambda_i \in \F_q} \frac{q^{a_i}-1}{q-1}$. 
\end{theorem}

\begin{proof} The number of eigenvectors of $Q$ up to scaling by $\F_q^\times$ with eigenvalue $\lambda_i$ is $\#\mathbb{P}^{a_i-1}(\F_q)$, the number of elements in the span of $a_i$ linearly independent vectors, up to scaling. Since fixed points correspond to eigenvectors with eigenvalues in $\F_q$ up to scaling, the number of fixed points is the total number of eigenvectors in $\mathbb{P}^{m-1}(\F_q)$ with eigenvalue in $\F_q$, which is the sum of $\# \mathbb{P}^{a_i-1}(\F_q)=\frac{q^{a_i}-1}{q-1}$ for all $i$ with $\lambda_i \in \F_q$. Then, we have that the number of fixed points of $\tau_Q$ is $$\sum_{i:\lambda_i \in \F_q} \frac{q^{a_i}-1}{q-1}.$$ \qedhere 
\end{proof} 
Fripertinger \cite{HF} computed the cycle structure as follows: let $\varphi(x)\in \F_q[x]$. Define $\exp(\varphi)$ to be the smallest positive integer such that $\varphi \mid (x^{\exp(\varphi)}-1)$. Similarly, define $\subexp(\varphi)$ to be the smallest positive integer such that $\varphi \mid (x^{\subexp(\varphi)}-\alpha)$ for some $\alpha \in \F_q^\times$. Note that $\subexp(\varphi)=\frac{\exp(\varphi)}{\gcd(q-1, \exp(\varphi))}$. 

Now, let $\varphi=x^d+b_{d-1}x^{d-1} + \cdots + b_1x+b_0 \in \F_q[x]$ be monic and irreducible, and let $$C(\varphi)\colonequals \left(
\begin{array}{cccccc}
0 & 0 & \cdots & 0 & 0 & -b_0\\
1 & 0 & \cdots & 0 & 0 & -b_1\\
0 & 1  &\cdots & 0 & 0 & -b_2\\
\vdots & \vdots & \ddots & \vdots& \vdots & \vdots\\
0 & 0 & \cdots & 1 & 0 & -b_{d-2}\\
0 & 0 & \cdots & 0 & 1 & -b_{d-1}
\end{array} \right).$$ 
Define $E_{1d}\colonequals(e_{ij})_{i,j}$, where $$e_{ij}=
\begin{cases}
1, & \text{if } i,j=1,d;\\
0,  & \text{otherwise}.
\end{cases}.$$
Then, the \emph{hypercompanion matrix} of $\varphi(x)^k$, for $k \in \Z_{> 0}$ is $$H(\varphi^k)=\left(
\begin{array}{cccccc}
C(\varphi) & 0 & 0&\cdots & 0&0\\
E_{1d} & C(\varphi) & 0& \cdots &0&0\\
0 & E_{1d} & C(\varphi) &\cdots & 0 &0\\
\vdots& \vdots& \vdots& \ddots& \vdots&\vdots\\
0 & 0& 0&\cdots & C(\varphi) & 0\\
0 & 0 & 0& \cdots & E_{1d} &C(\varphi)\\
\end{array}\right).$$

\begin{theorem} \label{thm4}
Letting $f_j=\subexp(\varphi^j)$, with the notation above, the number of $\ell$-cycles in $\sigma_{H(\varphi^k)}$ is $$\sum_{j: f_j=\ell; j \leq k} \frac{q^{jd}-q^{jd-d}}{(q-1)\ell}.$$ 
\end{theorem}

\begin{proof}
See \cite[Theorem 4]{HF} and the preceding discussion.
\end{proof}

\begin{lemma}
Every matrix $Q \in \GL_m(\F_q)$ is conjugate to a block diagonal matrix whose diagonal blocks are hypercompanion matrices of $\varphi^k$ for some monic irreducible $\varphi$ dividing the minimal polynomial of $Q$.
\end{lemma}

\begin{proof} View $\F_q^m$ as an $\F_q[x]$-module, by $f \cdot v=f(Q)v$, and the corresponding elementary divisor decomposition. This gives a basis $\beta$ such that the matrix for the transformation given by $Q$, with respect to the basis $\beta$ is a block diagonal matrix whose blocks are $H(\varphi^k)$ for each elementary divisor $\varphi^k$. 
\end{proof}

A product formula is given for the cycle structure of block diagonal matrices. For a more complete description of the cycle structure of $\tau_Q$, see Fripertinger \cite[Theorem 4]{HF} and the preceding discussion about the cycle structure given by the action of $\GL_m(\F_q)$ on $\mathbb{P}^{m-1}(\F_q)$, where the details of the action are worked out. In the simple case where the characteristic and minimal polynomial of $Q$ are equal, and a power of an irreducible polynomial, then we have the following result.

\begin{theorem}
If $\charpoly(Q)=\minpoly(Q)=\varphi^k$, where $\varphi$ is irreducible, then $Q \sim H(\varphi^k)$, and the number of $\ell$-cycles of $\tau_Q$ is given by the coefficient of $x^\ell$ in the formula $$\sum_{j\leq k} \frac{q^{jd}-q^{(j-1)d}}{(q-1) f_j}x^{f_j},$$  where $f_j=\subexp(\varphi^j)=\subexp(\varphi)p^t$ where $t=\min\{r \in \Z_{\geq 0}: p^r \geq j\}$. 
\end{theorem}

\begin{proof}

Follows immediately from Theorem \ref{thm4}.
\end{proof}

From this, we recover the result given by Forsyth--Gurev--Shrima in \cite{FGS} for the case when $n=2$. 

\begin{theorem}\label{thm1} Let $Q \in \GL_2(\F_q)$ act on $\mathbb{P}^1(\F_q)$ and let $\ell$ be the smallest positive integer such that $Q^\ell \in \F_q^\times$. Then, all points that are not fixed by $\tau_Q$ lie in an $\ell$-cycle.
\end{theorem}

\begin{proof} If the minimal polynomial of $Q \in \GL_2(\F_q)$ is $m_Q(x)=\varphi(x)^k$, where $\varphi$ is irreducible over $\F_q$ ($k=1$ or 2), then $Q \sim H(\varphi^k)$ and all points are either fixed or lie in a cycle of size $\subexp(\varphi^k)$. Otherwise, $m_Q(x)=(x-a)(x-b)$ and $Q$ is conjugate to a diagonal matrix whose diagonal entries are $a$ and $b$. In this case, we only obtain new eigenvectors by raising $Q$ to the $\ell$th power, where $\ell$ is the smallest positive integer such that $a^\ell=b^\ell$; then $Q^\ell$ is scalar and all points that were not previously fixed lie in an $\ell$-cycle. \qedhere 

\end{proof}

If $m \not= 2$, but $Q$ is diagonalizable over the algebraic closure $\overline{\F_q}$, we have the following conditions under which all cycles of $\tau_Q$ that are not fixed points are the same size.

\begin{theorem} Let $Q \in \GL_m(\F_q)$ be a diagonalizable matrix over $\overline{\F_q}$ with eigenvalues $\lambda_1, \dots, \lambda_s$. Then the permutation $\tau_Q$ contains only fixed points and $\ell$-cycles for some $\ell \in \Z_{>0}$ if and only if $\lambda_i/\lambda_j$ has multiplicative order $\ell$ in $\F_q$ for all $\lambda_i \not= \lambda_j$. In this case, $\ell$ is the smallest positive integer such that $Q^\ell=\alpha$ for some $\alpha \in \F_q$; i.e. $Q^{\ell}$ is a scalar matrix. 

\end{theorem}

\begin{proof}
First, note that fixed points of a matrix $Q$ correspond to eigenvectors of $Q$. Let $E_{\lambda_i}$ be the eigenspace of eigenvalue $\lambda_i$ in $\mathbb{P}^{m-1}(\F_q)$. Each of the $E_{\lambda_i}$ subspaces are disjoint and their span is $\mathbb{P}^{m-1}(\F_q)$. 

Suppose that $\tau_Q$ has only fixed points and $\ell$-cycles. Note that $\ell$ is the smallest positive integer such that $Q^\ell=\alpha$ for some $\alpha \in \F_q^\times$, since $\sigma_{Q^\ell}=\id$ if and only if $Q^\ell$ is a scalar matrix. Then, given $i \not= j$, $i,j \leq m$, we will show that $\lambda_i^{\ell_0}\not=\lambda_j^{\ell_0}$. Let $v_i \in E_i$ and $v_j \in E_j$. Clearly, $\lambda_i^\ell=\lambda_j^\ell=\alpha$ and this does not hold for any $\ell_0<\ell$, otherwise $v_i+v_j$ would be a fixed point of $Q^{\ell_0}$ and hence would lie in a cycle of size $\ell_0 <\ell$. Therefore, $(\lambda_i/\lambda_j)^\ell$ has multiplicative order $\ell$ in $\overline{\F_q}$. 

Conversely, suppose that $\lambda_i/\lambda_j$ has multiplicative order $\ell$ for all $i \not= j$, so $\lambda_i^\ell=\alpha$ for all $i$, and for some $\alpha \in \overline{\F_q}$. All cycles have size dividing $\ell$ for $Q^\ell=\alpha$ is scalar, and $\sigma_{Q^\ell}=\id$. Now, suppose that there is an element $v \in \mathbb{P}^{m-1}(\F_q)$ that lies in a cycle of size $\ell_0 < \ell$.  Then, write $v=\sum_{i=1}^m c_iv_i \in \mathbb{P}^{m-1}(\F_q)$, where $v_i \in E_{\lambda_i}$. If $c_i=0$ for all but one $i$, call it $i_0$ then $v$ is fixed by $Q$, since it is in $E_{\lambda_{i_0}}$. So, suppose $c_i, c_j \not= 0$ for some $i \not= j$. Then, suppose $Q^{\ell_0}v=\lambda ' v$ for some $\lambda '$, and $\ell_0<\ell$. Then, $\lambda '=\lambda_i^{\ell_0}=\lambda_j^{\ell_0}$, so $\lambda_i/\lambda_j$ has multiplicative order dividing $\ell_0<\ell$, a contradiction. \qedhere 

\end{proof} 


We now have a way of determining the permutation induced by an element in $\omega \in \Or$ by switching the order of factorization of its product with an element of prime norm, if $\Or$ is a PIR. A further question that arises from this problem is: what are the minimal conditions we must require of our order $\Or$ in order to obtain the same connection between the order an a matrix ring obtained by reduction modulo its Jacobson radical. This would involve the theory of Eichler or Hereditary orders.

\end{document}